\documentclass{article}%
\usepackage{amsfonts}
\usepackage{amsmath}%
\usepackage{amssymb}%
\usepackage{graphicx}
\newtheorem{theorem}{Theorem}

\newtheorem{corollary}[theorem]{Corollary}

\newtheorem{proposition}[theorem]{Proposition}

\newenvironment{proof}[1][Proof]{\noindent\textbf{#1.} }{\ \rule{0.5em}{0.5em}}
\begin{document}

\title{Leavitt path algebras which are Zorn rings}
\author{Kulumani M. Rangaswamy\\Department of Mathematics, University of Colorado\\Colorado Springs, Colorado 80918, USA}
\date{}
\maketitle

\begin{abstract}
Let $E$ be an arbitrary graph and $K$ be any field. It is shown that the
Leavitt path algebra $L_{K}(E)$ is a Zorn ring if and only if the graph $E$
satisfies Condition (L). As a consequence, every homomorphic image of
$L_{K}(E)$ is a Zorn ring if and only if $L_{K}(E)$ is a weakly regular ring.
The corresponding statement for the graph C*-algebra $C^{\ast}(E)$ is also investigated.

\end{abstract}

\section{Introduction and Preliminaries}

The notion of a Leavitt path algebra of a graph $E$ over a field $K$ was
introduced in \cite{AA}, \cite{AMP} as an algebraic analogue of a graph
C$^{\ast}$-algebra and the study of its algebraic structure has been the
subject of a series of papers in recent years (see, for e.g., \cite{AA},
\cite{ABCR}, \cite{ARS}, \cite{AMMS}, \cite{APS}, \cite{T}). Two specific
properties of the graph $E$ that play an important role in the investigation
of the Leavitt path algebra $L_{K}(E)$ are Condition (K) and Condition (L) of
the graph $E$ (see their definition below). It has been shown that Condition
(K) on a graph $E$ is equivalent to a number of ring-theoretical properties of
the Leavitt path algebra $L_{K}(E)$ such as being an exchange ring, weakly
regular, or having every ideal graded etc. (see \cite{ARS}, \cite{AMMS},
\cite{APS}, \cite{G}, \cite{T} ). What about the weaker Condition (L) ?
\ Condition (L) of the graph $E$ plays an important role in the fundamental
Cuntz-Krieger uniqueness theorem for C*-algebras (see \cite{KPR}) and in the
corresponding uniqueness theorem for Leavitt path algebras. Here we
investigate the possible ring-theoretic properties of the Leavitt path algebra
$L_{K}(E)$ which are equivalent to the graph $E$ satisfying Condition (L).
Specifically, we show that the graph $E$ satisfies Condition (L) if and only
if $L_{K}(E)$ is a Zorn ring (see its definition below). We also obtain
several equivalent formulations of a Zorn ring. In general, a homomorphic
image of a Zorn ring need not be a Zorn ring. As an easy consequence of our
main theorem, we show that every homomorphic image of $L_{K}(E)$ is a Zorn
ring if and only if $L_{K}(E)$ is a weakly regular ring.

It is an interesting fact that often, for a given graph $E$, the Leavitt path
algebra $L_{K}(E)$ has a certain algebraic property ( such as being a simple
ring, a purely inifinite simple ring, an exchange rings etc.,) if and only the
corresponding graph C*-algebra $C^{\ast}(E)$ has the same property ( see e.g.
\cite{AA}, \cite{APS}, \cite{T}, \cite{BPRS}, \cite{KPR}). But the methods to
establish these properties are usually independent of each other, with
algebraic arguments for $L_{K}(E)$ and independent analytical arguments for
$C^{\ast}(E)$. In line with these investigations, we combine our algebraic
methods with analytical arguments to show that if the C*-algebra $C^{\ast}(E)$
for graph $E$ is a Zorn ring, then necessarily $E$ satisfies Condition (L).
Conversely, if $E$ satisfies Condition (L) then every one-sided open ideal of
$C^{\ast}(E)$ is shown to contain an idempotent.

A ring $R$ is called a \textit{Zorn ring} if given an element $a\in R$, either
$a$ is nilpotent or there is an element $b\in R$ such that $ab$ is a non-zero
idempotent. In this case, $baba$ is also a non-zero idempotent. So this
definition is right/left symmetric. Also, since $abab$ is also an idempotent,
we see that $b$ can be chosen so that both $ab$ and $ba$ are idempotents. As
we shall see later, there are other reformulations of the definition of a Zorn ring.

The concept of a Zorn ring was introduced for alternate rings by M. Zorn
\cite{Z}. I. Kaplansky \cite{Kap2} named the rings with this property Zorn
rings, as did N. Bourbaki \cite{B}. Associative Zorn rings have had several
incarnations: In \cite{L} and \cite{N} , they were studied under the name
\textit{I-rings}. In \cite{Kap1}, Kaplansky renamed them weakly regular rings.
But during the last forty years, weakly regular rings have been named for
rings $R$ with a different property, namely, $I^{2}=I$ for every one-sided
ideal $I$ of $R$ (see, for instance, \cite{VSR}, \cite{ARS}). So we prefer to
use the name Zorn rings for these rings.

Zorn rings are among the class of rings with a large supply of idempotents
such as von Neumann regular rings. If a ring $R$ is von Neumann regular, then
it is always a Zorn ring, but the converse is not true as is clear by
considering the ring $R=%
%TCIMACRO{\U{2124} }%
%BeginExpansion
\mathbb{Z}
%EndExpansion
/p^{n}%
%TCIMACRO{\U{2124} }%
%BeginExpansion
\mathbb{Z}
%EndExpansion
$ of intergers modulo $p^{n}$ where $p$ is a prime and $n\geq2$. Leavitt path
algebras which are von Neumann regular or $\pi$-regular have been
characterized in \cite{AR}.

All the graphs $E$ that we consider here are arbitrary in the sense that no
restriction is placed either on the number of vertices in $E$ (such as being a
countable graph) or on the number of edges emitted by any vertex (such as
being a row-finite graph). \ We shall follow \cite{AA}, \cite{T} for the
general notation, terminology and results. For the sake of completeness, we
shall outline some of the concepts and results that we will be using.

A (directed) graph $E=(E^{0},E^{1},r,s)$ consists of two sets $E^{0}$ and
$E^{1}$ together with maps $r,s:E^{1}\rightarrow E^{0}$. The elements of
$E^{0}$ are called \textit{vertices} and the elements of $E^{1}$
\textit{edges}.

A vertex $v$ is called a \textit{regular vertex} if $s^{-1}(v)$ is a finite
non-empty set. A path $\mu$ in a graph $E$ is a finite sequence of edges
$\mu=e_{1}\dots e_{n}$ such that $r(e_{i})=s(e_{i+1})$ for $i=1,\dots,n-1$.
Such a path $\mu$ is \textit{closed} if $r(e_{n})=s(e_{1})$, in which case
$\mu$ is said to be\textit{\ based at} the vertex $s(e_{1})$. A closed path
$\mu$ as above is called \textit{simple} provided it does not pass through its
base more than once, i.e., $s(e_{i})\neq s(e_{1})$ for all $i=2,...,n$. The
closed path $\mu$ is called a \textit{cycle} if it does not pass through any
of its vertices twice, that is, if $s(e_{i})\neq s(e_{j})$ for every $i\neq
j$. An \textit{exit }for a path $\mu=e_{1}\dots e_{n}$ is an edge $e$ such
that $s(e)=s(e_{i})$ for some $i$ and $e\neq e_{i}$.

We say that a graph $E$ satisfies \textit{Condition }(L) if every simple
closed path in $E$ has an exit. $E$ is said to satisfy \textit{Condition
}(K\textit{)} provided no vertex $v\in E^{0}$ is the base of precisely one
simple closed path, i.e., either no simple closed path is based at $v$, or at
least two are based at $v$. Condition (K) on a graph $E$ always implies
Condition (L), but not coversely.

For each $e\in E^{1}$, we call $e^{\ast}$ a ghost edge. We let $r(e^{\ast}) $
denote $s(e)$, and we let $s(e^{\ast})$ denote $r(e)$.

Given an arbitrary graph $E$ and a field $K$, the \textit{Leavitt path }%
$K$\textit{-algebra }$L_{K}(E)$ is defined to be the $K$-algebra generated by
a set $\{v:v\in E^{0}\}$ of pairwise orthogonal idempotents together with a
set of variables $\{e,e^{\ast}:e\in E^{1}\}$ which satisfy the following conditions:

(1) $s(e)e=e=er(e)$ for all $e\in E^{1}$.

(2) $r(e)e^{\ast}=e^{\ast}=e^{\ast}s(e)$\ for all $e\in E^{1}$.

(3) For all $e,f\in E^{1}$, $e^{\ast}e=r(e)$ and $e^{\ast}f=0$ if $e\neq f$.

(4) For every regular vertex $v\in E^{0}$,
\[
v=\sum_{e\in E^{1},s(e)=v}ee^{\ast}.
\]

If $\mu=e_{1}\dots e_{n}$ is a path in $E$, we denote by $\mu^{\ast}$ the
element $e_{n}^{\ast}\dots e_{1}^{\ast}$ of $L_{K}(E)$.

A useful observation is that every element $a$ of $L_{K}(E)$ can be written as
$a=%
%TCIMACRO{\tsum \limits_{i=1}^{n}}%
%BeginExpansion
{\textstyle\sum\limits_{i=1}^{n}}
%EndExpansion
k_{i}\alpha_{i}\beta_{i}^{\ast}$, where $k_{i}\in K$, $\alpha_{i},\beta_{i}$
are paths in $E$ and $n$ is a suitable integer (see \cite{AA}).

\section{Leavitt path algebras and C*-algebras which are Zorn rings}

The next theorem describes the Leavitt path algebras which are Zorn rings. The
equivalence of conditions (ii) and (iii) of Theorem 1 is known \cite{Kap1}. We
include that proof for the sake of completeness.

\begin{theorem}
Let $E$ be an arbitrary graph and $K$ be any field. Then the following
properties are equivalent:

(i) \ $L_{K}(E)$ is a Zorn ring.

(ii) Every non-zero right/left ideal of $L_{K}(E)$ of contains a non-zero idempotent.

(iii) For every non-zero $a\in L_{K}(E)$, there is a non-zero $b\in L_{K}(E)$
such that $bab=b$.

(iv) The graph $E$ satisfies Condition (L).
\end{theorem}

\begin{proof}
Assume (i). Let $I$ be a non-zero right ideal of $L_{K}(E)$. Since the
Jacobson radical $J(L_{K}(E))=0$, $I$ is not a nil ideal. Let $a\in I$ be an
element which is not nilpotent. By hypothesis, there is a $b\in L_{K}(E)$ such
that $ab$ is a non-zero idempotent. Clearly $ab\in I$. A similar arguments
works if $I$ is a left ideal. This proves (ii).

Assume (ii). Let $0\neq a\in L_{K}(E)$. By hypothesis, $aR$ contains a
non-zero idempotent $ax$ for some $x\in R$. Then $ax=axaxax$. Multiplying \ by
$x$ on the left, we get $xax=(xax)a(xax)$. Clearly $xax\neq0$ since $xax=0$
implies $ax=axax=a(xax)=0$, a contradiction. A similar argument holds for the
left ideal $Ra$. This proves (iii).

Assume (iii). First observe that, for any idempotent $u$, the corner
$uL_{K}(E)u$ also satisfies condition (iii). To see this, let $a=uau\in
uL_{K}(E)u$. By hypothesis, there is a $b\in L_{K}(E)$ so that $buaub=b$.
Multiplying by $u$ on both sides, we get ($ubu)uau(ubu)=ubu$, thus proving the
desired conclusion. We \ wish to show that the graph $E$ satisfies Condition
(L). Suppose, by way of contradiction, there exists a cycle $c$ without exits
based at a vertex $v$ in $E$. As shown in Lemma 1.5 of \cite{AMMS}, an
examination of the elements of $vL_{K}(E)v$ leads to an isomorphism
$\varphi:vL_{K}(E)v\longrightarrow K[x,x^{-1}]$ mapping $v$ to $1$, $c$ to $x$
and $c^{\ast}$ to $x^{-1}$. We then obtain a contradiction, because while the
corner $vL_{K}(E)v$ satisfies condition (iii), $K[x,x^{-1}] $ does not, due to
the fact that for any $a\neq0$ in the integral domain $K[x,x^{-1}]$, $bab=b$
for some non-zero $b$ implies that $ab=1$, a contradiction. Hence $E$
satisfies Condition (L). This proves (iv).

Assume (iv). Let $a$ be a non-zero element of $L_{K}(E)$. By Proposition 3.1
of \cite{AMMS}, there are elements $x,y\in L_{K}(E)$ such that $xay=v$, a
vertex in $E$. Multiplying by $v$ on both sides, we may assume that $vx=x$ and
$yv=y$. Then the element $a(yx)$ is an idempotent since $ayxayx=ayvx=ayx$.
Also $ayx\neq0$ since $ayx=0$ implies that $0=xayx=vx=x$ and this will then
imply that $v=xay=0$, a contradiction. Hence $L_{K}(E)$ is Zorn ring, thus
proving (i).
\end{proof}

REMARK: (i) The proof of (iv) =%
%TCIMACRO{\TEXTsymbol{>} }%
%BeginExpansion
$>$
%EndExpansion
(i) shows that if a Leavitt path algebra $L_{K}(E)$ is a Zorn ring, it
actually satisfies a stronger condition: For \textit{every }non-zero element
$a$ in $L_{K}(E)$ (even if $a$ is nilpotent), there exists an element $b$ such
that $ab\neq0$ is an idempotent. This is to be expected since the Jacobson
radical $J(L_{K}(E))=0$ and so $aL_{K}(E)$ is not a nil right ideal.

(ii) It is also clear from the proof of Theorem 1 that $L_{K}(E)$ is a Zorn
ring if and only if every corner $eL_{K}(E)e$ is a Zorn ring, the converse
following from the fact that, in a Leavitt path algebra, every element $a$
belongs some corner $eL_{K}(E)e$ where the idempotent $e$ is a finite sum of
suitable finitely many vertices.

\bigskip

In general, a primitive ring or even a simple ring with identity need not be a
Zorn ring: Let $R$ be a simple ring without zero divisors which is not a
division ring (For instance, consider the example of J. Cozzens \cite{C} of a
simple right/left principal ideal ring with 1 and without zero divisors but
not a division ring). This ring is not a Zorn ring as $R$ cannot have
idempotents other than $0,1$. But for a Leavitt path algebra, the implication
holds. Because, it was shown in \cite{ABR} that if a Leavitt path algebra
$L_{K}(E)$ is a primitive ring, then necessarily the graph $E$ must satisfy
Condition (L). So from Theorem 1 we get the following corollary.

\begin{corollary}
If a Leavitt path algebra $L_{K}(E)$ is a primitive ring, then $L_{K}(E)$ is a
Zorn ring.
\end{corollary}

A homomorphic image of a Zorn ring $R$ need not, in general, be a Zorn ring
even if $R$ is a Leavitt path algebra. For example, let $R=L_{K}(E)$ be the
Toeplitz algebra (where $E$ is the graph with a single loop $c$ based at a
vertex $v$ with an exit $e$ at $v)$. Since $E$ satisfies Condition (L),
Theorem 1 implies that $R$ is a Zorn ring. But if $I$ is the ideal generated
by $e$, then $R/I$ is not a Zorn ring since $R/I\cong L_{K}(F)$ where
$F=E\backslash\{e\}$ is the graph with a single loop $c$ based at a vertex $v
$ and obviously does not satisfy Condition (L). The next Proposition describes
when every homomorphic image of $L_{K}(E)$ will be a Zorn ring.

First, we need some preliminaries on graphs: A subset $H$ of $E^{0}$ is called
\textit{hereditary} if whenever $v\in H$ and there is a path from $v$ to a
vertex $w$, then $w\in H$. A hereditary set is \textit{saturated} if, for any
regular vertex $v$, $r(s^{-1}(v))\subseteq H$ implies $v\in H$. Let $H$ be a
hereditary saturated subset of $E^{0}$. A vertex $w$ is called a
\textit{breaking vertex }of $H$ if $w\in E^{0}\backslash H$ is an infinite
emitter with the property that $1\leq|s^{-1}(w)\cap r^{-1}(E^{0}\backslash
H)|<\infty$. The set of all breaking vertices of $H$ is denoted by $B_{H}$.
For any $v\in B_{H}$, $v^{H}$ denotes the element $v-\sum_{s(e)=v,r(e)\notin
H}ee^{\ast}$. Given a hereditary saturated subset $H$ and a subset $S\subseteq
B_{H}$, $(H,S)$ is called an \textit{admissible pair }and $I_{(H,S)}$ denotes
the ideal generated by $H\cup\{v^{H}:v\in S\}$. Given an admissible pair
$(H,S)$, the corresponding \textit{quotient graph} $E\backslash(H,S)$ is
defined as follows:%
\begin{align*}
(E\backslash(H,S))^{0} &  =(E^{0}\backslash H)\cup\{v^{\prime}:v\in
B_{H}\backslash S\};\\
(E\backslash(H,S))^{1} &  =\{e\in E^{1}:r(e)\notin H\}\cup\{e^{\prime}:e\in
E^{1},r(e)\in B_{H}\backslash S\}.
\end{align*}
Further, $r$ and $s$ are extended to $E\backslash(H,S)$ by setting
$s(e^{\prime})=s(e)$ and $r(e^{\prime})=r(e)^{\prime}$. Theorem 5.7 of
\cite{T} (see also \cite{ABCR}) states that, given an admissible pair $(H,S)$,
$L_{K}(E)/I_{(H,S)}\cong L_{K}(E\backslash(H,S))$

Recall that a ring $R$ is said to be \textit{weakly regular} if $I^{2}=I$ for
every one sided ideal $I$ of $R$, equivalently, each element $a$ in $R$
satisfies $a\in aRaR$ ($a\in RaRa$). Note that being weak regular and being a
Zorn ring are independent properties as justified by the Toeplitz algebra and
the Cozzen's example. However, for Leavitt path algebras, we have an
interesting relationship between them, as indicated in the following Proposition.

\begin{proposition}
Let $E$ be an arbitrary graph and $K$ be any field. Then every homomorphic
image of the Leavitt path algebra $L_{K}(E)$ is a Zorn ring if and only if
$L_{K}(E)$ is a weakly regular ring.
\end{proposition}

\begin{proof}
Suppose every homomorphic image of $L_{K}(E)$ is a Zorn ring. By Theorem 1,
$E$ satisfies Condition (L). Now, for any hereditary saturated subset $H$ of
$E^{0}$ and any subset $S$ of $B_{H}$, $L_{K}(E\backslash(H,S))\cong
L_{K}(E)/I_{(H,S)}$, by Theorem 5.7 of \cite{T}. Hence $L_{K}(E\backslash
(H,S))$ is a Zorn ring and this implies, by Theorem 1, that the graph
$E\backslash(H,S))$ satisfies Condition (L). Since this happens for every
admissible pair $(H,S)$, we appeal to Proposition 6.12 of \cite{T} to conlude
that $E$ satisfies Condition (K). Then, by Theorem 3.15 of \cite{ARS},
$L_{K}(E)$ is weakly regular.

Conversely, suppose $L_{K}(E)$ is weakly regular. Now the weak regularity of
$L_{K}(E)$ implies, by Theorem 3.15 of \cite{ARS}, that the graph $E$
satisfies Condition (K). Since Condition (K) implies Condition (L), we
conclude by Theorem 1 that $L_{K}(E)$ is a Zorn ring. Since a homomorphic
image of a weakly regular ring is again weakly regular, we conclude that every
homomorphic image of $L_{K}(E)$ is a Zorn ring.
\end{proof}

\bigskip

REMARK: The above Proposition does not hold for arbitrary rings. For instance,
let $R=%
%TCIMACRO{\U{2124} }%
%BeginExpansion
\mathbb{Z}
%EndExpansion
/p^{n}%
%TCIMACRO{\U{2124} }%
%BeginExpansion
\mathbb{Z}
%EndExpansion
$, the ring of integers modulo $p^{n}$ where $p$ is a prime and $n$ is an
integer $\geq2$. Then $R\supset pR\supset......\supset p^{n-1}R\supset
p^{n}R=\{0\}$ are all the ideals of $R$ and, for each $i$ with $1\leq i\leq n
$, $R/p^{i}R$ is a Zorn ring since each element in it is either nilpotent or a
unit. But $R$ is not weakly regular since for the ideal $I=pR$, $I^{2}\neq I$.

GRAPH C*-ALGEBRAS: Next we show that, for any graph $E$, if the C*-algebra
$C^{\ast}(E)$ is a Zorn ring, then the graph $E$ satisfies Condition (L). For
the notation and terminology used in graph C*-algebras see \cite{RAE}.

In Theorem 1 above, the algebraic conditions (i) and (iii) are shown to be
equivalent for any ring. It is easily seen that this is valid for $C^{\ast
}(E)$ also. We shall use the fact that there is a $^{\ast}$-monomorphism
$\varphi:L_{%
%TCIMACRO{\U{2102} }%
%BeginExpansion
\mathbb{C}
%EndExpansion
}(E)\longrightarrow C^{\ast}(E)$ such that $\varphi(v)=p_{v}$, $\varphi
(\alpha\beta^{\ast})=s_{\alpha}s_{\beta^{\ast}}$ for all paths $\alpha,\beta$
in $E$ and that $L_{%
%TCIMACRO{\U{2102} }%
%BeginExpansion
\mathbb{C}
%EndExpansion
}(E)\longrightarrow\varphi(L_{%
%TCIMACRO{\U{2102} }%
%BeginExpansion
\mathbb{C}
%EndExpansion
}(E))=A$ is actually a graded isomorphism. Here $%
%TCIMACRO{\U{2102} }%
%BeginExpansion
\mathbb{C}
%EndExpansion
$ is the field of complex numbers. Moreover $C^{\ast}(E)=\bar{A}$, the
completion of $A$ under the induced norm.

Suppose $C^{\ast}(E)$ is a Zorn ring. We claim that $E$ satisfies Condition
(L). Suppose, on the contrary, there is a cycle $c$ without exits based at a
vertex $v$ in $E$. As noted in the proof of Theorem 1, $vL_{%
%TCIMACRO{\U{2102} }%
%BeginExpansion
\mathbb{C}
%EndExpansion
}(E)v\cong%
%TCIMACRO{\U{2102} }%
%BeginExpansion
\mathbb{C}
%EndExpansion
\lbrack x,x^{-1}]$. Since multiplication is continuous, $p_{v}C^{\ast}%
(E)p_{v}=p_{v}\bar{A}p_{v}=\overline{p_{v}Ap_{v}}\cong\overline{vL_{%
%TCIMACRO{\U{2102} }%
%BeginExpansion
\mathbb{C}
%EndExpansion
}(E)v}\cong\overline{%
%TCIMACRO{\U{2102} }%
%BeginExpansion
\mathbb{C}
%EndExpansion
\lbrack x,x^{-1}]}\cong%
%TCIMACRO{\U{2102} }%
%BeginExpansion
\mathbb{C}
%EndExpansion
(T)$, the algebra of complex valued continuous functions on $T$, where $T$ is
the unit circle consisting of complex numbers with modulus 1. (The last
isomorphism is well-known among C*-algebraists and Mark Tomforde writes that
it may also be obtained by combining Lemma 2.4 of \cite{aHR} with Theorem
4.1(c) of \cite{BPRS}). The composite of the above isomorphisms maps $p_{v}$to
$1$, $s_{c}$ to the identity function $z$ and $c^{\ast}$ to $z^{-1}$. Our
hypothesis implies that corresponding to the element $p_{v}-s_{c}$, there is
an element $b\in C^{\ast}(E)$ such that $b(p_{v}-s_{c})b=b$. Multiplying by
$p_{v}$ on both sides, we may assume that $b=p_{v}bp_{v}$ so that both
$(p_{v}-s_{c}),b\in p_{v}C^{\ast}(E)p_{v}$. In $%
%TCIMACRO{\U{2102} }%
%BeginExpansion
\mathbb{C}
%EndExpansion
(T)$, the equation $b(p_{v}-s_{c})b=b$ becomes $(1-z)f^{2}=f$ where $f\in%
%TCIMACRO{\U{2102} }%
%BeginExpansion
\mathbb{C}
%EndExpansion
(T)$ is some continuous function which is the image of $b$ under the given
isomorphism. This leads to a contradiction for the following reasons:

If $f(z)\neq0$ for all $z\in T$, then $f(z)=1/(1-z)$ for all $z\in T$ , but
this is not possible since $f$ is continuous.

So $f(z)=0$ for at least one $z\in T$. Consider the function $|f|$ which is
continuous since $f$ is. When $f(z)=0$, $|f(z)|=0$. For all $z\in T$ for which
$f(z)\neq0$, we have $f(z)=1/(1-z)$ and $|f(z)|=1/|(1-z)|\geq1/(1+|z|)=1/2$,
as $z$ is on the unit circle. This implies that $im(|f(z|)$ is disconnected
and hence $|f|$ is not continuous, a contradiction. Hence the graph must
satisfy Condition (L).

Thus we have proved the following:

\begin{proposition}
Let $E$ be an arbitrary graph and $K$ be any field. If the graph C*-algebra
$C^{\ast}(E)$ is a Zorn ring, then the graph $E$ satisfies Condition (L).
\end{proposition}

What about the Converse of Theorem 3?. Suppose the graph $E$ satisfies
Condition (L). By the Cuntz-Krieger uniqueness theorem \cite{KPR}, every
two-sided ideal of $C^{\ast}(E)$ contains a vertex. By Theorem 1 above, $L_{%
%TCIMACRO{\U{2102} }%
%BeginExpansion
\mathbb{C}
%EndExpansion
}(E)$ is a Zorn ring and since $L_{%
%TCIMACRO{\U{2102} }%
%BeginExpansion
\mathbb{C}
%EndExpansion
}(E)$ is a dense subalgebra of $C^{\ast}(E)$, we conclude that $C^{\ast}(E) $
is the completion of a Zorn ring. Thus every one-sided ideal of $C^{\ast}(E)$
which is open contains a non-zero idempotent. I do not know if $C^{\ast}(E)$
itself is a Zorn ring in this case.

After this paper was written, I learnt from Gene Abrams and Mercedes Silas
Molina that, in the book on Leavitt path algebras that they are currently
writing, they have also independently proved the equivalence of condition (ii)
and (iv) of Theorem 1 using different methods.

\end{document}